\documentclass[11pt,reqno]{amsart}
\usepackage{geometry}
\usepackage{color}
\usepackage{amssymb}
\usepackage{amsmath}
\usepackage{arydshln}

\def\BibTeX{{\rm B\kern-.05em{\sc i\kern-.025em b}\kern-.08em
    T\kern-.1667em\lower.7ex\hbox{E}\kern-.125emX}}

\newtheorem{thm}{Theorem}[section]

\newtheorem{lem}[thm]{Lemma}
\newtheorem{prop}[thm]{Proposition}

\newtheorem*{claim*}{Claim}

\theoremstyle{definition}
\newtheorem{defn}{Definition}[section]

\numberwithin{equation}{section}

\usepackage{hyperref}
\usepackage{bbm}
\usepackage{mathrsfs}

\usepackage{pdflscape}

\begin{document}
\title[Regularity of transition densities and ergodicity for affine jump-diffusions]{Regularity of transition densities and ergodicity for affine jump-diffusion processes}

\author{Martin Friesen}
\address[Martin Friesen]{School of Mathematics and Natural Sciences\\ University of Wuppertal\\ 42119 Wuppertal, Germany}
\email{friesen@math.uni-wuppertal.de}

\author[Peng Jin]{Peng Jin}
\address[Peng Jin]{Division of Science and Technology\\ BNU-HKBU United International College\\ Zhuhai, China}
\email{pengjin@uic.edu.cn}

\author{Jonas Kremer}
\address[Jonas Kremer]{School of Mathematics and Natural Sciences\\ University of Wuppertal\\ 42119 Wuppertal, Germany}
\email{kremer@math.uni-wuppertal.de}

\author{Barbara R\"udiger}
\address[Barbara R\"udiger]{School of Mathematics and Natural Sciences\\ University of Wuppertal\\ 42119 Wuppertal, Germany}
\email{ruediger@uni-wuppertal.de}

\date{\today}



\begin{abstract}
In this paper we study the transition density and exponential ergodicity in total variation for an affine process on the canonical state space $\mathbb{R}_{\geq0}^{m}\times\mathbb{R}^{n}$.
Under a H\"ormander-type condition for diffusion components as well as a boundary non-attainment assumption, we derive the existence and regularity of the transition density for the affine process and then prove the strong Feller property.
Moreover, we also show that under these and the additional subcritical conditions the corresponding affine process on the canonical state space is exponentially ergodic in the total variation distance.
To prove existence and regularity of the transition density we derive some precise estimates for the real part of the characteristic function of the process.
Our ergodicity result is a consequence of a suitable
application of a Harris-type theorem based on a local Dobrushin condition combined with the regularity of the transition densities.
\end{abstract}

\maketitle
\allowdisplaybreaks

\section{Introduction and main results}
\subsection{Affine processes}

The general notion of affine processes was first introduced by Duffie,
Filipovi{\'{c}}, and Schachermayer \cite{MR1994043} (2003) and
it provides a unified treatment of Ornstein-Uhlenbeck type processes
on $\mathbb{R}^{n}$ and CBI (continuous-state branching processes
with immigration) processes on $\mathbb{R}_{\geq0}^{m}$. Such processes
have been widely used in mathematical finance. In the following we
will recall the framework of affine processes on the canonical state
space, mainly following \cite{MR1994043}. Denote by $D:=\mathbb{R}_{\geq0}^{m}\times\mathbb{R}^{n}$
the \emph{canonical state space}, where $m,\thinspace n\in\mathbb{N}_{0}$
with $m+n>0$. For $D$, we write $I=\{1,\ldots,m\}$ and $J=\{m+1,\ldots,m+n\}$
for the index sets of the $\mathbb{R}_{\geq0}^{m}$-valued components
and the $\mathbb{R}^{n}$-valued components, respectively. For $x\in D$,
let $x_{I}=(x_{i})_{i\in I}$ and $x_{J}=(x_{j})_{j\in J}$. Throughout
this paper, we use the notation
\begin{align}
A=\begin{pmatrix}A_{II} & A_{IJ}\\
A_{JI} & A_{JJ}
\end{pmatrix}\label{eq: decomposition matrix}
\end{align}
for a $d\times d$-matrix $A$, where $A_{II}=(a_{ij})_{i,j\in I}$,
$A_{IJ}=(a_{ij})_{i\in I,j\in J}$, $A_{JI}=(a_{ij})_{i\in J,j\in I}$,
and $A_{JJ}=(a_{ij})_{i,j\in J}$. We endow $D$ with the usual inner
product $\langle\cdot,\cdot\rangle$ and denote by $\Vert x\Vert$
the induced Euclidean norm of a vector $x\in D$. Finally, let $\mathbb{S}_{d}^{+}$
stand for the cone of symmetric positive semidefinite $d\times d$-matrices.

\begin{defn}\label{def:admissible parameters} We call $(a,\alpha,b,\beta,m,\mu)$
a \emph{set of admissible parameters} for the state space $D$ if

(i) $a\in\mathbb{S}_{d}^{+}$ and $a_{kl}=0$ for $k\in I$ or $l\in I$;

(ii) $\alpha=(\alpha_{1},\ldots,\alpha_{m})$ with $\alpha_{i}=(\alpha_{i,kl})_{1\leq k,l\leq d}\in\mathbb{S}_{d}^{+}$
and $\alpha_{i,kl}=0$ if $k\in I\backslash\{i\}$ or $l\in I\backslash\{i\}$;

(iii) $\nu$ is a Borel measure on $D\backslash\{0\}$ satisfying

\[
\int_{D\backslash\{0\}}\left(1\wedge\left\Vert \xi\right\Vert ^{2}+\sum_{i\in I}\left(1\wedge\xi_{i}\right)\right)\nu(\mathrm{d}\xi)<\infty;
\]

(iv) $\mu=(\mu_{1},\ldots,\mu_{m})$ where every $\mu_{i}$ is a Borel
measure on $D\backslash\{0\}$ satisfying
\begin{equation}
\int_{D\backslash\{0\}}\left(\left\Vert \xi\right\Vert \wedge\left\Vert \xi\right\Vert ^{2}+\sum_{k\in I\backslash\{i\}}\xi_{k}\right)\mu_{i}\left(\mathrm{d}\xi\right)<\infty;\label{cond: conservative}
\end{equation}

(v) $b\in D$;

(vi) $\beta\in\mathbb{R}^{d\times d}$ with $\beta_{ki}-\int_{D\backslash\{0\}}\xi_{k}\mu_{i}(\mathrm{d}\xi)\geq0$
for all $i\in I$ and $k\in I\setminus\{i\}$, and $\beta_{IJ}=0$.
\end{defn}

In contrast to the definition of admissible parameters of Duffie \textit{et
al.} \cite{MR1994043}, here we neglect parameters corresponding to
killing and our definition includes an additional first moment condition
on the jump measures $\mu_{i}$. Let
\[
\mathcal{U}=\mathbb{C}_{\leq0}^{m}\times\mathbb{R}^{n}=\left\lbrace u=(u_{I},u_{J})\in\mathbb{C}^{m}\times\mathbb{C}^{n}\thinspace:\thinspace\mathrm{Re}(u_{I})\leq0,\thinspace\mathrm{Re}(u_{J})=0\right\rbrace .
\]
Note that the function $D\ni x\mapsto\exp(\langle u,x\rangle)$ is
bounded for any $u\in\mathcal{U}$. We denote the Banach space of
continuous functions that vanish at infinity by $C_{0}(D)$. Moreover,
$C_{c}^{2}(D)$ stands for the function space of twice continuously
differentiable functions on $D$ with compact support and $C_{c}^{\infty}(D)$
for the space of smooth functions on $D$ with compact support, respectively.

\begin{thm}[\cite{MR1994043}]\label{thm:characterization of affine processes}
Let $(a,\alpha,b,\beta,m,\mu)$ be a set of admissible parameters.
Then there exists a unique conservative Feller transition semigroup
$(P_{t})_{t\geq0}$ acting on $C_{0}(D)$ such that its infinitesimal
generator $(\mathcal{A},\mathrm{dom}\thinspace\mathcal{A})$ satisfies
$C_{c}^{2}(D)\subset\mathrm{dom}\thinspace\mathcal{A}$ and, for $f\in C_{c}^{2}(D)$
and $x\in D$,
\begin{align*}
\mathcal{A}f(x) & =\langle b+\beta x,\nabla f(x)\rangle+\sum_{k,l=1}^{d}\left(a_{kl}+\sum_{i=1}^{m}\alpha_{i,kl}x_{i}\right)\frac{\partial^{2}f(x)}{\partial x_{k}\partial x_{l}}\\
 & \quad+\int_{D\backslash\lbrace0\rbrace}\left(f(x+\xi)-f(x)-\langle\xi,\nabla_{J}f(x)\rangle\mathbbm{1}_{\lbrace\Vert\xi\Vert\leq1\rbrace}\right)\nu(\mathrm{d}\xi)\\
 & \quad+\sum_{i=1}^{m}x_{i}\int_{D\backslash\lbrace0\rbrace}\left(f(x+\xi)-f(x)-\langle\xi,\nabla f(x)\rangle\right)\mu_{i}(\mathrm{d}\xi),
\end{align*}
where $\nabla_{J}=(\partial/(\partial_{j}x_{j}))_{j\in J}$. Moreover,
$C_{c}^{\infty}(D)$ is a core for the generator and the Fourier transform
of the transition semigroup $(P_{t})_{t\geq0}$ has the representation
\begin{equation}
\int_{D}\mathrm{e}^{\langle u,\xi\rangle}P_{t}(x,\mathrm{d}\xi)=\exp\left(\phi(t,u)+\langle x,\psi(t,u)\rangle\right),\quad t\geq0,\thinspace u\in\mathcal{U},\label{eq:affine property}
\end{equation}
where $\phi(t,u)$ and $\psi(t,u)=(\psi_{I}(t,u),\psi_{J}(t,u))$
solve the generalized Riccati differential equations, for each $u=(u_{I},u_{J})\in\mathcal{U}$,
\begin{align}
\begin{cases}
\partial_{t}\phi(t,u) & =F(\psi(t,u)),\quad\phi(0,u)=0,\\
\partial_{t}\psi_{I}(t,u) & =R\left(\psi\left(t,u\right)\right),\quad\psi_{I}\left(0,u\right)=u_{I}\\
\psi_{J}(t,u) & =\mathrm{e}^{\beta_{JJ}^{\top}t}u_{J}
\end{cases}\label{eq:riccati eq for psi}
\end{align}
and the function $R$ and $F$ are given by
\begin{align*}
F(u) & =\langle u,au\rangle+\langle b,u\rangle+\int_{D\backslash\{0\}}\left(\mathrm{e}^{\langle u,\xi\rangle}-1-\langle u_{J},\xi_{J}\rangle\mathbbm{1}_{\lbrace\Vert\xi\Vert\leq 1\rbrace}\left(\xi\right)\right)\nu\left(\mathrm{d}\xi\right),\\
R_{i}(u) & =\langle u,\alpha_{i}u\rangle+\sum_{k=1}^{d}\beta_{ki}u_{k}+\int_{D\backslash\{0\}}\left(\mathrm{e}^{\langle u,\xi\rangle}-1-\langle u,\xi\rangle\right)\mu_{i}\left(\mathrm{d}\xi\right),\quad i\in I.
\end{align*}
\end{thm}

In view of \cite[Lemma 9.2]{MR1994043}, by the first moment condition
on $\mu_{i}$ given in Definition \ref{def:admissible parameters}
and the absence of the parameters according to killing, the transition
semigroup $(P_{t})_{t\geq0}$ with admissible parameters $(a,\alpha,b,\beta,\nu,\mu)$
is indeed conservative. Note also that in Theorem \ref{thm:characterization of affine processes}
the form of $\mathcal{A}f$ looks slightly different compared with
the one given in \cite[Theorem 2.7, formula (2.12)]{MR1994043}. In
fact, we used different compensation in the integrand of the integral
with respect to $\nu$ and $\mu_{i}$. However, we may modify the
drift parameters $b$ and $\beta$ accordingly so that the presentation
provided here is actually equivalent to that of \cite{MR1994043}.

\begin{defn} A conservative Markov process with state space $D$
and with transition semigroup $(P_{t})_{t\geq0}$ given in Theorem
\ref{thm:characterization of affine processes} is called an \emph{affine
process} with admissible parameters $(a,\alpha,b,\beta,\nu,\mu)$.
\end{defn}

We refer to \cite{MR1994043,2019arXiv190105815F,MR3313754} for extensive
surveys on the developments of affine processes. The purpose of this
work is twofold: on the one hand, we investigate sufficient conditions
for the regularity of the transition density including the strong
Feller property. On the other hand, based on the latter results, we
show that $P_{t}(x,\cdot)$ converges in total variation exponentially
fast to its unique invariant measure.

\subsection{Existence and regularity of transition densities}

For given $n\times n$-matrices $A_{1},\dots,A_{n}$ we write $[A_{1},\ldots,A_{n}]$
for the $n\times n^{2}$-block matrix which is obtained by putting
the matrices next to each other. Let us then introduce the $n\times n^{2}$-matrix
$\mathcal{K}$ given by
\begin{equation}
\mathcal{K}=\left[a_{JJ},\beta_{JJ}a_{JJ},\dots,\beta_{JJ}^{n-1}a_{JJ}\right],\label{eq:density condition}
\end{equation}
where we have used the notation \eqref{eq: decomposition matrix}.
For a given nonnegative integer $p$, let $C_{0}^{p}(D)$ be the subspace
of $C_{0}(D)$ of all $p$-times continuously differentiable functions
whose derivatives up to order $p$ all belong to $C_{0}(D)$. For
given multi-indices $\mathbf{q}=(q_{1},\dots,q_{d}),\ \tilde{\mathbf{q}}=(\tilde{q}_{1},\dots,\tilde{q}_{d})\in\mathbb{N}_{0}^{d}$
we denote by
\[
\partial_{(x,y)}^{(\mathbf{q},\tilde{\mathbf{q}})}=\frac{\partial^{|\mathbf{q}|+|\tilde{\mathbf{q}}|}}{\partial x_{1}^{q_{1}}\dots\partial x_{d}^{q_{d}}\partial y_{1}^{\tilde{q}_{1}}\dots\partial y_{d}^{\tilde{q}_{d}}}
\]
the corresponding mixed partial derivatives of orders $|\mathbf{q}|=q_{1}+\dots q_{d}$
and $|\tilde{\mathbf{q}}|=\tilde{q}_{1}+\dots+\tilde{q}_{d}$ acting
on functions in the variables $(x,y)\in D\times D$. The following
is our main result on the existence and regularity of the transition
densities.

\begin{thm}\label{thm:existence of densities} Assume that $\mathcal{K}$
given in \eqref{eq:density condition} has full rank, that $\min_{i\in\{1,\dots,m\}}\alpha_{i,ii}>0$,
and there exists a nonnegative integer $p$ with
\begin{align}
p<\min_{i\in I}\frac{b_{i}}{\alpha_{i,ii}}-m.\label{eq: condition boundary}
\end{align}
Then the following assertions hold:
\begin{enumerate}
\item[(a)] For each $x\in D$ and $t>0$, $P_{t}(x,\cdot)$ admits a density
$f_{t}(x,\cdot)$ of class $C_{0}^{p}(D)$ given by
\begin{equation}
f_{t}(x,y)=\int_{\mathbb{R}^{d}}\mathrm{e}^{-\langle y,\mathrm{i}u\rangle}\mathrm{e}^{\phi(t,\mathrm{i}u)+\langle x,\psi(t,\mathrm{i}u)\rangle}\frac{\mathrm{d}u}{(2\pi)^{d}},\qquad y\in D.\label{eq:densitiy of p_t}
\end{equation}
The function $(0,\infty)\times D\times D\ni(t,x,y)\mapsto f_{t}(x,y)$
defined by \eqref{eq:densitiy of p_t} is jointly continuous and,
for each $t>0$, the mapping $x\mapsto f_{t}(x,\cdot)\in L^{1}(D)$
is continuous.
\item[(b)] For each pair of multi-indices $\mathbf{q},\tilde{\mathbf{q}}\in\mathbb{N}_{0}^{d}$
satisfying
\begin{align}
q_{1}=\dots=q_{m}=0\ \text{ and }\ \sum_{i=1}^{m}\tilde{q}_{i}\leq p,\label{eq: multi index b}
\end{align}
the derivative $\partial_{(x,y)}^{(\mathbf{q},\tilde{\mathbf{q}})}f_{t}$
exists and is jointly continuous in $(t,x,y)\in(0,\infty)\times D\times D$.
Moreover, it holds that for all $t_{1}>t_{0}>0$
\[
\sup_{(t,x,y)\in[t_{0},t_{1}]\times D\times D}\left|\partial_{(x,y)}^{(\mathbf{q},\tilde{\mathbf{q}})}f_{t}(x,y)\right|<\infty.
\]
\item[(c)] Assume that
\begin{equation}
\int_{\left\{ 0<\left\Vert \xi\right\Vert \le1\right\} }\left\Vert \xi\right\Vert \mu_{i}\left(\mathrm{d}\xi\right)<\infty,\quad i=1,\ldots,m.\label{eq: addtion condtion on mu_i}
\end{equation}
Then for each pair of multi-indices $\mathbf{q},\tilde{\mathbf{q}}\in\mathbb{N}_{0}^{d}$
satisfying
\begin{align}
\sum_{i=1}^{m}(q_{i}+\tilde{q}_{i})\leq p,\label{eq: multi index c}
\end{align}
the derivative $\partial_{(x,y)}^{(\mathbf{q},\tilde{\mathbf{q}})}f_{t}$
exists for $x\in D^{\mathrm{o}}$, $y\in D$, $t>0$ and is jointly
continuous in $(t,x,y)\in(0,\infty)\times D^{\mathrm{o}}\times D$,
where $D^{\mathrm{o}}$ denotes the interior of $D$. Moreover, it
holds that for any $t_{1}>t_{0}>0$ and any compact set $K\subset D^{\mathrm{o}}$,
\[
\sup_{(t,x,y)\in[t_{0},t_{1}]\times K\times D}\left|\partial_{(x,y)}^{(\mathbf{q},\tilde{\mathbf{q}})}f_{t}(x,y)\right|<\infty.
\]
\end{enumerate}
\end{thm}

Note that the continuity of $x\mapsto f_{t}(x,\cdot)\in L^{1}(D)$
stated in part (a) implies that the corresponding affine process has
the strong Feller property. The restrictions on the multi-indies $(\mathbf{q},\tilde{\mathbf{q}})$
imposed in part (b) assert that $(x,y)\mapsto f_{t}(x,y)$ is, for
$t>0$, smooth in $(x_{J},y_{J})$ and $p$-times continuously differentiable
in $y_{I}$. Part (c) asserts that if we assume finite first moment
for the small jumps of the measures $\mu_{i}$ and restrict ourselves
to the interior of $D$, then the transition density is also differentiable
with respect to the $x_{I}$ variables. Note that in all cases differentiability
with respect to the variables $x_{I}$, $y_{I}$ is determined by
the index $p$ given by \eqref{eq: condition boundary}, while the
density is smooth in the variables $x_{J}$, $y_{J}$.

The proof of Theorem \ref{thm:existence of densities} is given in
Section \ref{sec:ex and reg of denisities} and is essentially based
on an estimate on the characteristic function of the affine process
shown in Section 2. As a consequence of the proof, we find that the
derivatives $\partial_{(x,y)}^{(\mathbf{q},\tilde{\mathbf{q}})}f_{t}$
can be obtained from \eqref{eq:densitiy of p_t} by differentiating
under the integral. To derive the desired estimates on the characteristic
function of the affine process, we extend the methods developed in
\cite{MR3084047} and \cite{2019arXiv190805473F} to general affine
processes on the canonical state space $D$. Below we shall comment
on the conditions imposed in the above statements.

First, note that $\min_{i\in\{1,\dots,m\}}\alpha_{i,ii}>0$ implies
that the underlying diffusion component restricted to $\mathbb{R}_{\geq0}^{m}$
is non-degenerate except on the boundary $\partial\mathbb{R}_{\geq0}^{m}$,
which can be seen from the particular structure of the generator $\mathcal{A}$.
Condition \eqref{eq: condition boundary} is a multi-dimensional version
of the Feller condition for diffusions. Such condition guarantees
that the affine process does not charge the boundary of its state
space and that its law has a density being continuous on the whole
$D=\mathbb{R}_{\geq0}^{m}\times\mathbb{R}^{n}$ including the boundary
$\partial\mathbb{R}_{\geq0}^{m}\times\mathbb{R}^{n}$. Finally, note
that without condition \eqref{eq: condition boundary} we already
find that in dimension $d=m=1$ an affine process whose law has a
density on $(0,\infty)$ with $\{0\}$ being a singular point, see,
e.g., the remark given after \cite[Theorem 4.1]{MR3084047}.

Second, the assumption that $\mathcal{K}$ given by \eqref{eq:density condition}
has full rank guarantees that the underlying component on $\mathbb{R}^{n}$
drives the process to each point of the state space. Such condition
is an analogue of the well-known H\"ormander condition for diffusion
operators and is, for instance, satisfied if $a_{JJ}$ is invertible.
However, depending on the particular form of the drift $\beta_{JJ}$
one may also find examples where $a_{JJ}\neq0$ is not invertible.

While the $C_{0}^{p}(D)$ regularity of the density $f_{t}(x,\cdot)$
from statement (a) partially extends a known result as formulated
in \cite[Theorem 4.1]{MR3084047} for affine processes with state-dependent
jumps of finite variation, continuity and differentiability with respect
to the variable $x$ as discussed in statements (b) and (c) have not
been investigated yet in this generality. For affine processes with
state space $D=\mathbb{R}_{\geq0}$ (i.e., one-dimensional CBI processes)
existence and regularity of transition densities, including the strong
Feller property, was studied by analytic techniques in \cite{MR3805842}.
In \cite{MR3343292}, the strong Feller property of one-dimensional
CBI processes was established, provided Grey's condition is satisfied,
based on a precise analysis of the Laplace transform from which the
construction of a successful coupling was derived. Note that Grey's
condition is relatively mild and applies to diffusions as well as
processes with non-trivial jump behavior. Based on a precise convergence
rate for the one-step Euler approximations combined with a discrete
integration by parts formula, existence and Besov regularity of transition
densities were studied for multi-type CBI processes on state space
$D=\mathbb{R}_{\geq0}^{m}$, $m\geq1$, in \cite{FJR18}. Such a result
can be used to provide a simple proof of the strong Feller property
for multi-dimensional settings, see \cite{2019arXiv190805473F} where
the anisotropic stable JCIR process was studied. Certainly, this method
is also applicable to general affine processes through a combination
of the methods used in \cite{FJR18} with the stochastic equation
for affine processes derived in \cite{2019arXiv190105815F}. Since
one would only get a density in a weighted Besov space instead of
$C_{0}^{p}(D)$, the expected results would be less optimal than our
Theorem \ref{thm:existence of densities}. In contrast, this work
is restricted to affine processes for which the underlying diffusion
component is non-degenerate, while the aforementioned works investigate
models for which the diffusion part is allowed to be degenerate or
even absent.

\subsection{Exponential ergodicity in total variation}

We now turn to the investigation of the long-time behavior of affine
processes on $D$. Following \cite[Theorem 2.7]{2018arXiv181205402J},
for an affine transition kernel $P_{t}(x,\mathrm{d}\xi)$ on $D$
with an admissible parameter set $(a,\alpha,b,\beta,m,\mu)$ such
that $\beta$ has only eigenvalues with strictly negative real-parts,
and the state-independent jump-measure $\nu$ satisfies $\int_{\lbrace\Vert\xi\Vert>1\rbrace}\log\Vert\xi\Vert\nu(\mathrm{d}\xi)<\infty$,
there exists a unique invariant distribution\footnote{A probability measure $\pi$ on $D$ is said to be an invariant distribution
if
\[
\int_{D}P_{t}(x,\mathrm{d}\xi)\pi(\mathrm{d}\xi)=\pi(\mathrm{d}\xi),\quad t\geq0,\ \ x\in D.
\]
} $\pi$ and $P_{t}(x,\cdot)\to\pi$ weakly as $t\to\infty$. After
having established the existence of a unique invariant distribution
for affine processes, it is natural to study the convergence rate
towards this limiting distribution in different distances on the space
of probability measures. The particular choice of Wasserstein distances
was recently treated in \cite{2019arXiv190105815F} where an exponential
rate of convergence was established. Below we provide a similar statement
for the total variation distance.

For two measures $\varrho,\thinspace\widetilde{\varrho}$ on $\mathcal{P}(D)$,
the class of probability measures on $D$, the total variation distance
is defined by
\[
\left\Vert \varrho-\widetilde{\varrho}\right\Vert _{TV}:=\sup\left\lbrace \left\vert \varrho(A)-\widetilde{\varrho}(A)\right\vert \thinspace:\thinspace A\subset D\text{ Borel set}\right\rbrace .
\]

Our second main result is built on Theorem \ref{thm:existence of densities}
and yields the existence and regularity of the density for the unique
invariant measure $\pi$ as well as convergence of the transition
kernel towards $\pi$ in total variation.

\begin{thm}\label{thm:exp ergo} Let $P_{t}(x,\cdot)$ be the transition
semigroup of an affine process on $D$ with an admissible parameter
set $(a,\alpha,b,\beta,m,\mu)$. Suppose that $\beta$ has only eigenvalues
with strictly negative real-parts, $\int_{\lbrace\Vert\xi\Vert>1\rbrace}\log\Vert\xi\Vert m(\mathrm{d}\xi)<\infty$,
that $\mathcal{K}$ defined by \eqref{eq:density condition} has full
rank, $\min_{i\in\{1,\dots,m\}}\alpha_{i,ii}>0$, and \eqref{eq: condition boundary}
holds for some nonnegative integer $p$. Then the unique invariant
measure $\pi$ has a density $f^{\pi}$ of class $C_{0}^{p}(D)$.
Moreover, for each multi-index $\mathbf{q}\in\mathbb{N}_{0}^{d}$
satisfying $\sum_{i\in I}q_{i}\leq p$, the derivative $\partial^{|\mathbf{q}|}(f^{\pi})/\partial_{1}^{q_{1}}\dots\partial_{d}^{q_{d}}$
exists and is uniformly bounded. Finally, there exist constants $c,\thinspace C>0$
such that
\[
\Vert P_{t}(x,\cdot)-\pi(\cdot)\Vert_{TV}\leq C\mathrm{e}^{-ct}\left(1+\log\left(1+\Vert x\Vert\right)\right),\quad t\geq0,\thinspace x\in D.
\]
\end{thm} The proof of Theorem \ref{thm:exp ergo} is deferred to
Section \ref{sec:exp ergodicity}. Existence and regularity of the
density $f^{\pi}$ is a consequence of the estimates for the characteristic
function of the affine processes established in Section 2. The ergodicity
statement is based on a suitable application of a Harris-type theorem
combined with a coupling argument and follows some ideas taken from
\cite{2019arXiv190105815F}.

Existing literature on the exponential ergodicity in total variation
distance for affine processes on $\mathbb{R}_{\geq0}^{m}\times\mathbb{R}^{n}$
are often limited to specific affine models, see, e.g., \cite{MR3254346,2019arXiv190805473F,2019arXiv190202833F,MR3732190,JKR,MR3451177,MR3437080,MR3343292,MR2287102,MR3058984,2018arXiv181100122Z}.
Compared to Corollary \ref{thm:exp ergo} the most relevant result
in the literature are those of Zhang and Glynn \cite{2018arXiv181100122Z}
and of Mayerhofer, Stelzer and Vestweber\cite{2018arXiv181110542M},
where analogous resutls are obtained under similar conditions, but
with essentially more restrictive assumptions on the state-dependent
and state-independent jump measures. At this point it is worthwhile
to mention that, in addition to the ergodicity results the authors
also studied the central limit theorem in \cite{2018arXiv181100122Z}.
Also, in \cite{2018arXiv181110542M} the authors study affine processes
on symmetric cones which includes affine processes on $D=\mathbb{R}_{\geq0}^{m}$
as a particular case. Let us also mention that the ergodicity results
in the works \cite{2019arXiv190805473F,2019arXiv190202833F,MR3732190,MR3343292,MR2287102,MR3058984},
where the diffusion component is allowed to be degenerate, are limited
to particular classes of affine processes.

\subsection{Structure of the work}

In Section \ref{sec:estimates on the char func} we prove a pointwise estimate for the characteristic
function of the affine process on the canonical state space. Based
on this estimate we prove Theorem \ref{thm:existence of densities}
as well as the regularity of the density for the invariant measure
in Section \ref{sec:ex and reg of denisities}. Finally, ergodicity in the total variation distance
is shown in Section \ref{sec:exp ergodicity}.

\label{sec:transition density}

\section{Estimate on the characteristic function}
\label{sec:estimates on the char func}
\subsection{Main statement}

The following is our main result for this section. \begin{thm}\label{thm: estimate characteristic function}
Let $P_{t}(x,\cdot)$ be the transition probability of an affine process
on $D$ with an admissible parameter set $(a,\alpha,b,\beta,m,\mu)$.
Let $\phi$, $\psi$ be the corresponding solutions to \eqref{eq:riccati eq for psi}.
Assume that $\mathcal{K}$ given in \eqref{eq:density condition}
has full rank and that $\min_{i\in\{1,\dots,m\}}\alpha_{i,ii}>0$.
Then, for each $t_{0}>0$ and each $\vartheta>0$ there exist constants
$C_{t_{0},\vartheta},\thinspace M_{t_{0},\vartheta}, \thinspace\delta_{t_{0}, \vartheta}>0$
such that
\[
\mathrm{e}^{\mathrm{Re}(\phi(t,\mathrm{i}u))}\leq C_{t_{0},\vartheta}(1+\|u_{I}\|)^{-\lambda(\vartheta)}\mathrm{e}^{-\delta_{t_{0}, \vartheta}}\|u_{J}\|^{2}
\]
holds for all $u\in\mathbb{R}^{d}$ with $\|u\|\geq M_{t_{0},\vartheta}$
and all $t\geq t_{0}$, where
\[
\lambda(\vartheta)=\min_{i\in I}b_{i}/\widehat{\alpha}_{i,ii}(\vartheta)\geq0
\]
and
\begin{align}\label{eq: defi alpha hat-2}
\widehat{\alpha}_{i,ii}(\vartheta)=\alpha_{i,ii}+\int_{\lbrace\Vert\xi\Vert\leq\vartheta\rbrace}\Vert\xi\Vert^{2}\mu_{i}(\mathrm{d}\xi),\quad i\in I.
\end{align}
\end{thm}
Note that $\mathrm{Re}(\psi(t,\mathrm{i}u))\leq0$ holds
for $t\geq0$ and $u\in\mathbb{R}^{d}$. In particular, it holds that
\[
\left|\int_{D}\mathrm{e}^{\mathrm{i}\langle u,\xi\rangle}P_{t}(x,\mathrm{d}\xi)\right|=\left|\mathrm{e}^{\langle x,\psi(t,\mathrm{i}u)\rangle+\phi(t,\mathrm{i}u)}\right|\leq\mathrm{e}^{\mathrm{Re}(\phi(t,\mathrm{i}u))}.
\]
The rest of this subsection is devoted to the proof of Theorem \ref{thm: estimate characteristic function}.
It is obtained as a consequence of several technical lemmas proven
below.

\subsection{Technical lemmas}

Set
\[
U:=\left\{ \mathrm{Re}\thinspace u\thinspace:\thinspace u\in\mathcal{U}\right\} =\left\{ (x_{I},0)\in\mathbb{R}^{d}:x_{I}\in\mathbb{R}_{\leq0}^{m}\right\} .
\]
For $(x,y)\in U\times\mathbb{R}^{d}$ and $u\in\mathbb{R}^{d}\backslash\{0\}$,
we define
\begin{align*}
H_{1}^{i}(x,y;u) & =\langle x,\alpha_{i}x\rangle-\langle y,\alpha_{i}y\rangle+\frac{1}{\Vert u\Vert}\sum_{k=1}^{d}\beta_{ki}x_{k}\\
 & \quad+\frac{1}{\Vert u\Vert^{2}}\int_{D\backslash\lbrace0\rbrace}\left(\mathrm{e}^{\Vert u\Vert\langle\xi,x\rangle}\cos\left(\Vert u\Vert\langle\xi,y\rangle\right)-1-\Vert u\Vert\langle\xi,x\rangle\right)\mu_{i}(\mathrm{d}\xi),\\
H_{2}^{i}(x,y;u) & =2\langle x,\alpha_{i}y\rangle+\frac{1}{\Vert u\Vert}\sum_{k=1}^{d}\beta_{ki}y_{k}\\
 & \quad+\frac{1}{\Vert u\Vert^{2}}\int_{D\backslash\lbrace0\rbrace}\left(\mathrm{e}^{\Vert u\Vert\langle\xi,x\rangle}\sin\left(\Vert u\Vert\langle\xi,y\rangle\right)-\Vert u\Vert\langle\xi,y\rangle\right)\mu_{i}(\mathrm{d}\xi).
\end{align*}
These functions will serve as the vector fields for the Riccati equations
when decomposed into their real- and imaginary parts and rescaled
appropriately.

\begin{lem}\label{lem:continuity of H_1 and H_2} The functions $H_{k}^{i}(x,y;u)$,
$k=1,2$, are continuous in $(x,y,u)\in U\times\mathbb{R}^{d}\times\mathbb{R}^{d}\backslash\{0\}$
for each $i\in I$. \end{lem}

\begin{proof} Fix $i\in I$ and let $(x,y,u)\in U\times\mathbb{R}^{d}\times\mathbb{R}^{d}\backslash\{0\}$
be arbitrary. It suffices to prove continuity in $(x,y,u)$ of the
integrals with respect to $\mu_{i}$. We have
\begin{align*}
 & \left\vert \mathrm{e}^{\Vert u\Vert\langle\xi,x\rangle}\cos\left(\Vert u\Vert\langle\xi,y\rangle\right)-1-\Vert u\Vert\langle\xi,x\rangle\right\vert \\
 & \qquad\qquad\qquad\leq\left\vert \mathrm{e}^{\Vert u\Vert\langle\xi,x\rangle}\cos\left(\Vert u\Vert\langle\xi,y\rangle\right)-\mathrm{e}^{\Vert u\Vert\langle\xi,x\rangle}\right\vert +\left\vert \mathrm{e}^{\Vert u\Vert\langle\xi,x\rangle}-1-\Vert u\Vert\langle\xi,x\rangle\right\vert \\
 & \qquad\qquad\qquad\leq\mathbbm{1}_{\lbrace\Vert\xi\Vert\leq1\rbrace}(\xi)\Vert u\Vert^{2}\Vert\xi\Vert^{2}\left(\Vert x\Vert^{2}+\Vert y\Vert^{2}\right)+\mathbbm{1}_{\lbrace\Vert\xi\Vert>1\rbrace}(\xi)\left(2\Vert u\Vert\Vert\xi\Vert\Vert x\Vert+2\right),
\end{align*}
where we used that $x\in U$ and $\vert\cos(\Vert u\Vert\langle\xi,y\rangle)-1\vert=2\sin^{2}(\Vert u\Vert\langle\xi,y\rangle/2)\leq2\wedge\left(\Vert u\Vert^{2}\Vert\xi\Vert^{2}\Vert y\Vert^{2}\right)$.
In view of \eqref{cond: conservative}, we can apply dominated convergence
theorem to obtain the continuity of $H_{1}^{i}$. Similarly, we estimate
\begin{align*}
 & \left\vert \mathrm{e}^{\Vert u\Vert\langle\xi,x\rangle}\sin\left(\Vert u\Vert\langle\xi,y\rangle\right)-\Vert u\Vert\langle\xi,y\rangle\right\vert \\
 & \qquad\qquad\qquad\leq\left\vert \sin\left(\Vert u\Vert\langle\xi,y\rangle\right)-\Vert u\Vert\langle\xi,y\rangle\right\vert +\left\vert \mathrm{e}^{\Vert u\Vert\langle\xi,x\rangle}-1\right\vert \Vert u\Vert\vert\langle\xi,y\rangle\vert\\
 & \qquad\qquad\qquad=\left\vert \int_{0}^{1}\Vert u\Vert\langle\xi,y\rangle\cos\left(r\Vert u\Vert\langle\xi,y\rangle\right)\mathrm{d}r-\Vert u\Vert\langle\xi,y\rangle\right\vert +\left\vert \mathrm{e}^{\Vert u\Vert\langle\xi,x\rangle}-1\right\vert \Vert u\Vert\vert\langle\xi,y\rangle\vert\\
 & \qquad\qquad\qquad\leq\mathbbm{1}_{\lbrace\Vert\xi\Vert\leq1\rbrace}(\xi)\Vert u\Vert^{2}\Vert\xi\Vert^{2}\left(2\Vert y\Vert^{2}+\Vert x\Vert\Vert y\Vert\right)+\mathbbm{1}_{\lbrace\Vert\xi\Vert>1\rbrace}(\xi)\left(1+2\Vert u\Vert\Vert\xi\Vert\Vert y\Vert\right).
\end{align*}
Once again we can apply dominated convergence theorem to obtain the
continuity of $H_{2}^{i}$. \end{proof}

\begin{lem}\label{lem:key lemma for existence of a density} For
any $u\not=0$, let
\begin{equation}
F(t,u)=\frac{1}{\Vert u\Vert}\mathrm{Re}\thinspace\psi\left(\frac{t}{\Vert u\Vert},\mathrm{i}u\right)\quad\text{and}\quad G(t,u)=\frac{1}{\Vert u\Vert}\mathrm{Im}\thinspace\psi\left(\frac{t}{\Vert u\Vert},\mathrm{i}u\right).\label{eq: defi of F and G}
\end{equation}
Fix $i\in I$. Then for each $\varepsilon>0$ there exist some constants
$t_{0}>0$, $\varrho>0$, and $M>0$ such that
\[
F_{i}(t,u)\leq-\varrho t,\quad t\in(0,t_{0}],
\]
for all $u\in\mathbb{R}^{d}$ with $\Vert u\Vert\geq M$ and $\langle u,\alpha_{i}u\rangle\geq\varepsilon\Vert u\Vert^{2}$.
\end{lem}

\begin{proof} It is easy to see that $F$ and $G$ satisfy
\[
\begin{cases}
\partial_{t}F_{i}(t,u)=H_{1}^{i}(F(t,u),G(t,u);u), & F(0,u)=0\\
\partial_{t}G_{i}(t,u)=H_{2}^{i}(F(t,u),G(t,u);u), & G(0,u)=\frac{u}{\|u\|}.
\end{cases}
\]
Moreover, it holds $F_{J}\equiv0$ and $G_{J}(t,\mathrm{i}u)=\exp(\beta_{JJ}^{\top}t/\Vert u\Vert)u_{J}\Vert u\Vert^{-1}$.

\emph{Step 1}: Choose a small open neighbourhood $\widetilde{O}_{1}\subset\mathbb{R}^{d}$
of $0$, and also a large enough $M>0$ such that for $x\in O_{1}:=\widetilde{O}_{1}\cap\mathcal{U}$
and $\Vert u\Vert\ge M$,
\[
\left|\langle x,\alpha_{i}x\rangle+\frac{1}{\Vert u\Vert}\sum_{k=1}^{d}\beta_{ki}x_{k}+\frac{1}{\Vert u\Vert^{2}}\int_{D\backslash\lbrace0\rbrace}\left(\mathrm{e}^{\Vert u\Vert\langle\xi,x\rangle}-1-\Vert u\Vert\langle\xi,x\rangle\right)\mu_{i}(\mathrm{d}\xi)\right|\le\frac{\varepsilon}{4}.
\]

\emph{Step 2}: Define
\[
K:=\left\{ u\in\mathbb{R}^{d}:\langle u,\alpha_{i}u\rangle\geq\varepsilon,\ \ \|u\|=1\right\} .
\]
Then we can find an open neighbourhood $O_{2}$ of $K$ such that
\[
\langle y,\alpha_{i}y\rangle\geq\frac{\varepsilon}{2},\quad y\in O_{2}.
\]

\emph{Step 3}: Based on the last two steps, we now have, for $x\in O_{1}$,
$y\in O_{2}$, and $u\in\mathbb{R}^{d}$ satisfying $\Vert u\Vert\geq M$,
\begin{align*}
H_{1}^{i}(x,y;u) & =\langle x,\alpha_{i}x\rangle-\langle y,\alpha_{i}y\rangle+\frac{1}{\Vert u\Vert}\sum_{k=1}^{d}\beta_{ki}x_{k}\\
 & \quad+\frac{1}{\Vert u\Vert^{2}}\int_{D\backslash\lbrace0\rbrace}\left(\mathrm{e}^{\Vert u\Vert\langle\xi,x\rangle}\cos\left(\Vert u\Vert\langle\xi,y\rangle\right)-1-\Vert u\Vert\langle\xi,x\rangle\right)\mu_{i}(\mathrm{d}\xi)\\
 & \le-\langle y,\alpha_{i}y\rangle+\langle x,\alpha_{i}x\rangle+\frac{1}{\Vert u\Vert}\sum_{k=1}^{d}\beta_{ki}x_{k}\\
 & \quad+\frac{1}{\Vert u\Vert^{2}}\int_{D\backslash\lbrace0\rbrace}\left(\mathrm{e}^{\Vert u\Vert\langle\xi,x\rangle}-1-\Vert u\Vert\langle\xi,x\rangle\right)\mu_{i}(\mathrm{d}\xi)\\
 & \le-\frac{\varepsilon}{2}+\frac{\varepsilon}{4}=-\frac{\varepsilon}{4}.
\end{align*}

\emph{Step 4}: It is easy to verify that $H_{1}^{i}(x,y;u)$ and $H_{2}^{i}(x,y;u)$
are both bounded for $(x,y,\Vert u\Vert)\in O_{1}\times O_{2}\times[M,\infty)$,
i.e., there exists $\gamma>0$ such that
\[
|H_{1}^{i}(x,y;u)|+|H_{2}^{i}(x,y;u)|\leq\gamma,\quad\text{for all }(x,y,\Vert u\Vert)\in O_{1}\times O_{2}\times[M,\infty).
\]

\emph{Step 5}: Note that $O_{1}\times O_{2}$ is an open neighbourhood
of $\{0\}\times K\subset\mathcal{U}\times\mathbb{R}^{d}$ in the relative
topology w.r.t. $\mathcal{U}\times\mathbb{R}^{d}$. Since $\{0\}\times K$
is compact, the boundary of $O_{1}\times O_{2}$ (which is closed)
has a positive distance $l>0$ to the closed set $\{0\}\times K$.
So starting from a point in $\{0\}\times K$, running according to
the vector field $(H_{1}^{i},H_{2}^{i})$, we will have to wait at
least for a positive time $t_{0}:=l/\gamma$ to attain the boundary
of $O_{1}\times O_{2}$. At this point we implicitly use the fact
that $F_{I}(t,u)=\|u\|^{-1}\mathrm{Re}\ \psi(t/\|u\|,\mathrm{i}u)\leq0$.
Recall that for $(x,y,\Vert u\Vert)\in O_{1}\times O_{2}\times[M,\infty)$,
\[
H_{1}^{i}(x,y;u)\leq-\frac{\varepsilon}{4}.
\]
Hence, we obtain for $t\in(0,t_{0})$
\[
F_{i}(t,u)=\int_{0}^{t}\partial_{s}F_{i}(s,u)\mathrm{d}s=\int_{0}^{t}H_{1}^{i}(F(s,u),G(s,u);u)\mathrm{d}s\leq-\frac{\varepsilon}{4}t,
\]
for all $u\in\mathbb{R}^{d}$ with $\Vert u\Vert\ge M$ and $\langle u,\alpha_{i}u\rangle\geq\varepsilon\Vert u\Vert^{2}$.
\end{proof}

The remainder of our proof is motivated by \cite[Theorem 4.1]{MR3084047},
and we will extend it to the more general case where the L\'evy measures
$\mu=(\mu_{1},\dots,\mu_{m})$ do not have finite first moment for
the small jumps.

\begin{prop} Fix $i\in I$. Then for each $\varepsilon>0$,
$\vartheta>0$, and $t_{0}>0$, there exist some constants $M_{\varepsilon,t_{0}},\thinspace C_{\varepsilon,t_{0},\vartheta}>1$
such that
\begin{equation}
\exp\left(b_{i}\int_{0}^{t}\mathrm{Re}(\psi_{i}(s,\mathrm{i}u))\mathrm{d}s\right)\leq C_{\varepsilon,t_{0},\vartheta}\left(1+\Vert u\Vert\right)^{-b_{i}/\widehat{\alpha}_{i,ii}(\vartheta)},\qquad t\geq t_{0},\label{eq:integrability of the characteristic function}
\end{equation}
for all $u\in\mathbb{R}^{d}$ with $\Vert u\Vert\geq M_{\varepsilon,t_{0}}$
and $\langle u,\alpha_{i}u\rangle\geq\varepsilon\Vert u\Vert^{2}$.
\end{prop}

\begin{proof} Define $f(t,\mathrm{i}u)=\mathrm{Re}\thinspace\psi(t,\mathrm{i}u)$
and $g(t,\mathrm{i}u)=\mathrm{Im}\thinspace\psi(t,\mathrm{i}u)$.
Let $i\in I$ be fixed. Noting that $f_{J}(t,\mathrm{i}u)\equiv0$
and $g_{J}(t,\mathrm{i}u)=\mathrm{e}^{\beta_{JJ}^{\top}t}u_{J}$,
an easy computation shows that
\begin{align*}
\partial_{t}f_{i} & =\alpha_{i,ii}f_{i}^{2}-\langle g,\alpha_{i}g\rangle+\sum_{k=1}^{d}\widetilde{\beta}_{ki}f_{k}+\int_{D\backslash\lbrace0\rbrace}\left(\mathrm{e}^{\langle\xi,f\rangle}\cos\langle\xi,g\rangle-1-\xi_{i}f_{i}\right)\mu_{i}(\mathrm{d}\xi),\\
\partial_{t}g_{i} & =2f_{i}\alpha_{i,ii}g_{i}+\sum_{k=1}^{d}\beta_{ki}g_{k}+\int_{D\backslash\lbrace0\rbrace}\left(\mathrm{e}^{\langle\xi,f\rangle}\sin\langle\xi,g\rangle-\langle g,\xi\rangle\right)\mu_{i}(\mathrm{d}\xi)
\end{align*}
with initial conditions $f_{i}(0)=0$ and $g_{i}(0)=u_{i}$, where
\[
\widetilde{\beta}_{ki}=\begin{cases}
\beta_{ki}-\int_{D\backslash\lbrace0\rbrace}\xi_{k}\mu_{i}(\mathrm{d}\xi), & k\not=i\\
\beta_{ii}, & k=i.
\end{cases}
\]
Using that $f_{i}\leq0$, we have
\begin{align*}
 & \int_{D\backslash\lbrace0\rbrace}\left(\mathrm{e}^{\langle\xi,f\rangle}\cos\langle\xi,g\rangle-1-\xi_{i}f_{i}\right)\mu_{i}(\mathrm{d}\xi)\\
 & \quad\quad=\int_{D\backslash\lbrace0\rbrace}\left(\mathrm{e}^{\langle\xi,f\rangle}\cos\langle\xi,g\rangle-\mathrm{e}^{\xi_{i}f_{i}}\right)\mu_{i}(\mathrm{d}\xi)+\int_{D\backslash\lbrace0\rbrace}\left(\mathrm{e}^{\xi_{i}f_{i}}-1-\xi_{i}f_{i}\right)\mu_{i}(\mathrm{d}\xi)\\
 & \quad\quad\leq\int_{D\backslash\lbrace0\rbrace}\left(\mathrm{e}^{\xi_{i}f_{i}}-1-\xi_{i}f_{i}\right)\mu_{i}(\mathrm{d}\xi)\\
 & \quad\quad\leq f_{i}^{2}\int_{\lbrace\Vert\xi\Vert\leq\vartheta\rbrace}\xi_{i}^{2}\mu_{i}(\mathrm{d}\xi)-2f_{i}\int_{\lbrace\Vert\xi\Vert>\vartheta\rbrace}\xi_{i}\mu_{i}(\mathrm{d}\xi),
\end{align*}
yielding that
\[
\partial_{t}f_{i}\leq\widehat{\alpha}_{i,ii}(\vartheta)f_{i}^{2}-\langle g,\alpha_{i}g\rangle+\widehat{\beta}_{ii}(\vartheta)f_{i},
\]
where we define $\widehat{\beta}_{ii}(\vartheta)=\beta_{ii}-2\int_{\lbrace\Vert\xi\Vert>\vartheta\rbrace}\xi_{i}\mu_{i}(\mathrm{d}\xi)$
and we have used that
\[
 \widetilde{\beta}_{ki}=\beta_{ki}-\int_{D\backslash\{0\}}\xi_{k}\mu_{i}(\mathrm{d}\xi)\geq0, \qquad k \neq i.
\]
Moreover, for any $u\not=0$, we define the scaled
functions $F(t,u)$ and $G(t,u)$ as in Lemma \ref{lem:key lemma for existence of a density}.
Then $F_{i}$ satisfies the following differential inequality
\begin{equation}
\partial_{t}F_{i}\leq\widehat{\alpha}_{i,ii}(\vartheta)F_{i}^{2}-\langle G,\alpha_{i}G\rangle+\frac{1}{\Vert u\Vert}\widehat{\beta}_{ii}(\vartheta)F_{i}\leq\widehat{\alpha}_{i,ii}(\vartheta)F_{i}^{2}+\frac{1}{\Vert u\Vert}\widehat{\beta}_{ii}(\vartheta)F_{i}\label{eq: ineq for partial F}
\end{equation}
with initial condition $F_{i}(0)=0$, where we used that $\langle G,\alpha_{i}G\rangle\geq0$,
since $\alpha_{i}\in\mathbb{S}_{d}^{+}$. Applying Lemma \ref{lem:key lemma for existence of a density}
yields that for each $\varepsilon>0$ and $t_{0}>0$, there exist
constants $\delta_{\varepsilon,t_{0}}=\delta\in(0,t_{0})$, $\varrho_{\varepsilon,t_{0}}=\varrho>0$,
and $M_{\varepsilon,t_{0}}=M>0$ such that $F_{i}(\delta,u)\leq-\varrho$
for all $u\in\mathbb{R}^{d}$ with $\Vert u\Vert\geq M$ and $\langle u,\alpha_{i}u\rangle\geq\varepsilon\Vert u\Vert^{2}$.
Using \cite[Lemma C.3]{MR3084047}, we arrive at the following differential
inequality $F_{i}(t,u)\leq\overline{F}_{i}(t-\delta,u)$ for all $t\geq\delta$,
where $\overline{F}_{i}$ solves
\begin{align*}
\overline{F}_{i}(t,u) & =\widehat{\alpha}_{i,ii}(\vartheta)\overline{F}_{i}(t,u)^{2}+\frac{1}{\Vert u\Vert}\widehat{\beta}_{ii}(\vartheta)\overline{F}_{i}(t,u),\quad t>0,\\
\overline{F}_{i}(0,u) & =-\varrho.
\end{align*}
Without loss of generality we may and do assume that $\widehat{\beta}_{ii}(\vartheta)<0$;
indeed, if $\widehat{\beta}_{ii}(\vartheta)\ge0$, we can replace
$\widehat{\beta}_{ii}(\vartheta)$ in \eqref{eq: ineq for partial F}
by $-1$ and adjust the equation for $\overline{F}_{i}$ accordingly.
So
\[
F_{i}(t,u)\leq\overline{F}_{i}(t-\delta,u)=\frac{-\mathrm{e}^{\frac{\widehat{\beta}_{ii}(\vartheta)}{\Vert u\Vert}(t-\delta)}}{\Vert u\Vert\widehat{\alpha}_{i,ii}(\vartheta)\widehat{\beta}_{ii}(\vartheta)^{-1}\left(\mathrm{e}^{\frac{\widehat{\beta}_{ii}(\vartheta)}{\Vert u\Vert}(t-\delta)}-1\right)+\frac{1}{\varrho}},\qquad t\geq\delta,
\]
for all $u\in\mathbb{R}^{d}$ with $\Vert u\Vert\geq M$ and $\langle u,\alpha_{i}u\rangle\geq\varepsilon\Vert u\Vert^{2}$,
where we used \cite[Lemma C.5]{MR3084047} to derive an explicit solution
for $\overline{F}_{i}(t-\delta,u)$. Noting that $f_{i}(t,u)=\Vert u\Vert F_{i}(t\Vert u\Vert,\mathrm{i}u)$,
by rescaling and then integrating we get
\begin{align}
\int_{0}^{t}f_{i}(s,\mathrm{i}u)\mathrm{d}s & \leq\int_{\frac{\delta}{\Vert u\Vert}}^{t}f_{i}(s,\mathrm{i}u)\mathrm{d}s\nonumber \\
 & =\frac{-1}{\widehat{\alpha}_{i,ii}(\vartheta)}\log\left(\varrho\Vert u\Vert\widehat{\alpha}_{i,ii}(\vartheta)\widehat{\beta}_{ii}(\vartheta)^{-1}\left(\mathrm{e}^{\widehat{\beta}_{ii}(\vartheta)\left(t-\frac{\delta}{\Vert u\Vert}\right)}-1\right)+1\right),\label{eq:zwischenschritt}
\end{align}
for all $t\geq\delta$ (thus $t\ge\delta\ge\delta/M\geq\delta\Vert u\Vert^{-1}$)
and for all $u\in\mathbb{R}^{d}$ with $\Vert u\Vert\geq M$ and $\langle u,\alpha_{i}u\rangle\geq\varepsilon\Vert u\Vert^{2}$
(see the derivation of (C.8) in the proof of \cite[p.109, Theorem 4.1]{MR3084047}
for details on \eqref{eq:zwischenschritt}). Note that $\delta\in(0,t_{0})$
and $M>1$. This yields, for $t\geq t_{0}$ and $u\in\mathbb{R}^{d}$
with $\|u\|\geq M$ and $\langle u,\alpha_{i}u\rangle\geq\varepsilon\|u\|^{2}$,
\begin{align*}
\mathrm{e}^{b_{i}\int_{0}^{t}f_{i}(s,\mathrm{i}u)\mathrm{d}s} & \leq\left(1+\varrho\|u\|\widehat{\alpha}_{i,ii}(\vartheta)\widehat{\beta}_{ii}(\vartheta)^{-1}\left(\mathrm{e}^{\widehat{\beta}_{ii}(\vartheta)\left(t-\frac{\delta}{\Vert u\Vert}\right)}-1\right)\right)^{-b_{i}/\widehat{\alpha}_{i,ii}(\vartheta)}\\
 & \leq\left(1+\varrho\|u\|\widehat{\alpha}_{i,ii}(\vartheta)\widehat{\beta}_{ii}(\vartheta)^{-1}\left(\mathrm{e}^{\widehat{\beta}_{ii}(\vartheta)\left(t_{0}-\frac{\delta}{M}\right)}-1\right)\right)^{-b_{i}/\widehat{\alpha}_{i,ii}(\vartheta)}\\
 & \leq C_{\varepsilon,t_{0},\vartheta}\left(1+\|u\|\right)^{-b_{i}/\widehat{\alpha}_{i,ii}(\vartheta)},
\end{align*}
where $C_{\varepsilon,t_{0},\vartheta}>0$ is some constant independent
of $t$. This proves the assertion. \end{proof}

The next lemma provides a decomposition of the state space into cones
for which \eqref{eq:integrability of the characteristic function}
can be applied.

\begin{lem}\label{lemma: decomposition of state space} Let $t_{0}>0$
be arbitrary. Then there exists a constant $\varepsilon_{t_{0}}>0$
such that $\bigcup_{i=0}^{m}A_{i}=\mathbb{R}^{d}$, where
\begin{align}
A_{0}:=\left\{ u\in\mathbb{R}^{d}\ |\ \left\langle u_{J},\left(\int_{0}^{t_{0}}\mathrm{e}^{s\beta_{JJ}}a_{JJ}\mathrm{e}^{s\beta_{JJ}^{\top}}\mathrm{d}s\right)u_{J}\right\rangle \geq\varepsilon_{t_{0}}\|u\|^{2}\right\} \label{eq: definition C0}
\end{align}
and
\[
A_{i}=\{u\in\mathbb{R}^{d}\ |\ \langle u,\alpha_{i}u\rangle\geq\varepsilon_{t_{0}}\|u\|^{2}\},\qquad i\in\{1,\dots,m\}.
\]
\end{lem} \begin{proof} Let $t_{0}>0$ be fixed. Since the matrix
$\mathcal{K}$ defined by \eqref{eq:density condition} has full rank,
the same arguments as given in the proof of \cite[Lemma C.2]{MR3084047}
for the equivalence of (i) and (ii) also show that the matrix $\int_{0}^{t_{0}}\mathrm{e}^{s\beta_{JJ}}a_{JJ}\mathrm{e}^{s\beta_{JJ}^{\top}}\mathrm{d}s$
is non-singular and hence positive definite. So there exists large
enough $\kappa_{t_{0}}>1$ such that
\[
\kappa_{t_{0}}^{-1}\|u_{J}\|^{2}\le\langle u_{J},\left(\int_{0}^{t_{0}}\mathrm{e}^{s\beta_{JJ}}a_{JJ}\mathrm{e}^{s\beta_{JJ}^{\top}}\mathrm{d}s\right)u_{J}\rangle\le\kappa_{t_{0}}\|u_{J}\|^{2},\qquad u_{J}\in\mathbb{R}^{n}.
\]
Suppose, by contradiction, that $\bigcup_{i=0}^{m}A_{i}\neq\mathbb{R}^{d}$.
Then for each $\varepsilon>0$, there exists $u^{\varepsilon}=(u_{1}^{\varepsilon},\dots,u_{d}^{\varepsilon})\in\mathbb{R}^{d}$
such that $\langle u^{\varepsilon},\alpha_{i}u^{\varepsilon}\rangle<\varepsilon\|u^{\varepsilon}\|^{2}$
holds for all $i=1,\dots,m$ and
\[
\langle u_{J}^{\varepsilon},\left(\int_{0}^{t_{0}}\mathrm{e}^{s\beta_{JJ}}a_{JJ}\mathrm{e}^{s\beta_{JJ}^{\top}}\mathrm{d}s\right)u_{J}^{\varepsilon}\rangle<\varepsilon\|u^{\varepsilon}\|^{2}.
\]
This yields
\begin{align}
\min\{\kappa_{t_{0}}^{-1},\alpha_{1,11},\dots,\alpha_{m,mm}\}\|u^{\varepsilon}\|^{2} & \leq\min_{i\in\{1,\dots,m\}}\alpha_{i,ii}\|u_{I}^{\varepsilon}\|^{2}+\kappa_{t_{0}}^{-1}\|u_{J}^{\varepsilon}\|^{2}\nonumber \\
 & \leq\sum_{i=1}^{m}\langle u^{\varepsilon},\alpha_{i}u^{\varepsilon}\rangle+\langle u_{J}^{\varepsilon},\left(\int_{0}^{t_{0}}\mathrm{e}^{s\beta_{JJ}}a_{JJ}\mathrm{e}^{s\beta_{JJ}^{\top}}\mathrm{d}s\right)u_{J}^{\varepsilon}\rangle\nonumber \\
 & \leq\varepsilon\left(m+1\right)\|u^{\varepsilon}\|^{2}.\label{eq: contradiction}
\end{align}
Since $\min_{i\in\{1,\dots,m\}}\alpha_{i,ii}>0$ we may choose $\varepsilon>0$
small enough so that
\[
\varepsilon(m+1)<\min\left\{ \kappa_{t_{0}}^{-1},\alpha_{1,11},\ldots,\alpha_{m,mm}\right\} ,
\]
which contradicts \eqref{eq: contradiction}. \end{proof}

\begin{lem} For each $t_{0}>0$ it holds that
\begin{align}
\int_{0}^{t}\mathrm{Re}(\langle\psi(s,\mathrm{i}u),a\psi(s,\mathrm{i}u)\rangle)\mathrm{d}s\leq-\delta_{t_{0}}\|u_{J}\|^{2},\qquad u\in\mathbb{R}^{d},\ \ t\geq t_{0}\label{eq: bound on J component}
\end{align}
where
\[
\delta_{t_{0}}:=\inf_{\|u_{J}\|=1}\int_{0}^{t_{0}}\langle u_{J},\mathrm{e}^{s\beta_{JJ}}a_{JJ}\mathrm{e}^{s\beta_{JJ}^{\top}}u_{J}\rangle\mathrm{d}s>0.
\]
\end{lem} \begin{proof} Using first that $a_{II}=0$, $a_{IJ}=0$,
$a_{JI}=0$ and then $\psi_{J}(s,\mathrm{i}u)=\mathrm{e}^{s\beta_{JJ}^{\top}}\mathrm{i}u_{J}$
we find
\begin{align}
\int_{0}^{t}\langle\psi(s,\mathrm{i}u),a\psi(s,\mathrm{i}u)\rangle\mathrm{d}s & =-\int_{0}^{t}\langle u_{J},\mathrm{e}^{\beta_{JJ}s}a_{JJ}\mathrm{e}^{\beta_{JJ}^{\top}s}u_{J}\rangle\mathrm{d}s\leq-\delta_{t_{0}}\|u_{J}\|^{2}.\label{eq: new new new eq}
\end{align}
As shown in Lemma \ref{lemma: decomposition of state space}, the
matrix $\int_{0}^{t_{0}}\mathrm{e}^{s\beta_{JJ}}a_{JJ}\mathrm{e}^{s\beta_{JJ}^{\top}}\mathrm{d}s$
is positive definite which yields $\delta_{t_{0}}>0$. This proves
the assertion. \end{proof}

\subsection{Proof of Theorem \ref{thm: estimate characteristic function}}

We are now prepared to prove Theorem \ref{thm: estimate characteristic function}.
\begin{proof}[Proof of Theorem \ref{thm: estimate characteristic function}]
Fix $t_{0},\vartheta>0$ and let $x\in D$, $u\in\mathbb{R}^{d}$
be arbitrary. According to Lemma \ref{lemma: decomposition of state space},
there exists $\varepsilon_{t_{0}}>0$ such that $\bigcup_{i=0}^{m}A_{i}=\mathbb{R}^{d}$,
where $A_{i}=A_{i}(t_{0})$ is as in Lemma \ref{lemma: decomposition of state space}.
Using for $i=0$ equation \eqref{eq: new new new eq} and the definition
of $A_{0}$ in \eqref{eq: definition C0}, and for $i=1,\dots,m$
equations \eqref{eq:integrability of the characteristic function}
and \eqref{eq: bound on J component}, we find constants $C_{t_{0},\vartheta},\thinspace M_{t_{0},\vartheta}\geq1$
such that for all $\|u\|\geq M_{t_{0},\vartheta}$ and all $(t,x)\in(t_{0},\infty)\times D$,
\begin{align*}
 & \ \mathrm{e}^{\mathrm{Re}(\phi(t,\mathrm{i}u))}\\
 &\qquad =\sum_{i=0}^{m}\mathbbm{1}_{A_{i}}(u)\mathrm{e}^{\mathrm{Re}(\phi(t,\mathrm{i}u))}\\
 &\qquad \leq\mathbbm{1}_{A_{0}}(u)\mathrm{e}^{-\int_{0}^{t}\langle u_{J},\mathrm{e}^{\beta_{JJ}s}a_{JJ}\mathrm{e}^{\beta_{JJ}^{\top}s}u_{J}\rangle\mathrm{d}s}+C_{t_{0},\vartheta}\sum_{i=1}^{m}\mathbbm{1}_{A_{i}}(u)\left(1+\|u\|\right)^{-b_{i}/\widehat{\alpha}_{i,ii}(\vartheta)}\mathrm{e}^{-\delta_{t_{0}}\|u_{J}\|^{2}}\\
 &\qquad\leq\mathrm{e}^{-\varepsilon_{t_{0}}\|u\|^{2}}+C_{t_{0},\vartheta}m\left(1+\|u_{I}\|\right)^{-\lambda(\vartheta)}\mathrm{e}^{-\delta_{t_{0}}\|u_{J}\|^{2}}\\
 &\qquad \leq\widetilde{C}_{t_{0},\vartheta,m}\left(1+\|u_{I}\|\right)^{-\lambda(\vartheta)}\mathrm{e}^{-\delta}\|u_{J}\|^{2},
\end{align*}
where $\delta:=\varepsilon_{t_{0}}\wedge\delta_{t_{0}}$,
$\widetilde{C}_{t_{0},\vartheta,m}$ is a large enough constant and
we have used that $\mathrm{Re}(\psi(t,\mathrm{i}u))\leq0$ so that
\[
\mathrm{Re}(\phi(t,\mathrm{i}u))\leq\int_{0}^{t}\mathrm{Re}(\langle\psi(s,\mathrm{i}u),a\psi(s,\mathrm{i}u))\mathrm{d}s+\int_{0}^{t}\mathrm{Re}(\langle b,\psi(s,\mathrm{i}u)\rangle)\mathrm{d}s.
\]
This proves the assertion. \end{proof}

\section{Existence and regularity of densities}

\label{sec:ex and reg of denisities}

\subsection{Proof of Theorem \ref{thm:existence of densities}}
\begin{proof}[Proof of Theorem \ref{thm:existence of densities}]
(a) Using \eqref{eq:affine property}, then Theorem \ref{thm: estimate characteristic function}
and finally \eqref{eq: condition boundary}, we conclude that
\[
\sup_{(t,x)\in[t_{0},\infty)\times D}\int_{\mathbb{R}^{d}}\|u_{I}\|^{p}\|u_{J}\|^{q}\left\vert \int_{D}\mathrm{e}^{\langle\mathrm{i}u,\xi\rangle}P_{t}(x,\mathrm{d}\xi)\right\vert \mathrm{d}u<\infty.
\]
holds for each $t_{0}>0$ and $q\geq0$. By classical properties of
the Fourier transform, see \cite[Proposition 28.1]{MR3185174}, we
find that $P_{t}(x,\cdot)$ has a density $f_{t}(x,\cdot)$ given
by \eqref{eq:densitiy of p_t} and satisfies $f_{t}(x,\cdot)\in C_{0}^{p}(D)$
for each $t>0$ and $x\in D$. Since the integrand in \eqref{eq:densitiy of p_t}
is jointly continuous in $(t,x,y)\in(0,\infty)\times D\times D$,
by Theorem \ref{thm: estimate characteristic function} and the dominated
convergence theorem we also find that $D\times D\times(0,\infty)\ni(x,y,t)\mapsto f_{t}(x,y)$
is jointly continuous. Next, let us prove that $x\mapsto f_{t}(x,\cdot)\in L^{1}(D)$
is continuous for each $t>0$. So, let $t>0$ and $x\in D$ be fixed.
Let $(x_{n})_{n\in\mathbb{N}}\subset D$ so that $x_{n}\to x$ as
$n\to\infty$. We have
\begin{align*}
\limsup_{n\to\infty} & \left\Vert f_{t}(x_{n},\cdot)-f_{t}(x,\cdot)\right\Vert _{L^{1}(D)}\\
 & =\limsup_{n\to\infty}\int_{B_{\delta}(0)}\left\vert f_{t}(x_{n},y)-f_{t}(x,y)\right\vert \mathrm{d}y+\limsup_{n\to\infty}\int_{D\backslash B_{\delta}(0)}\left\vert f_{t}(x_{n},y)-f_{t}(x,y)\right\vert \mathrm{d}y,
\end{align*}
where $B_{\delta}(0)$ denotes the ball with center $0$ and radius
$\delta>0$. For the first integral on the right hand-side, by the
joint continuity of $f_{t}(x,y)$, we get
\[
\limsup_{n\to\infty}\int_{B_{\delta}(0)}\left\vert f_{t}(x_{n},y)-f_{t}(x,y)\right\vert \mathrm{d}y=0.
\]
Turning to the second integral, we note that $P_{t}(x_{n},\cdot)\to P_{t}(x,\cdot)$
weakly as $n\to\infty$ by \eqref{eq:affine property}. So we can
use Portmanteau's theorem to estimate
\begin{align*}
\limsup_{n\to\infty}\int_{D\backslash B_{\delta}(0)}\left\vert f_{t}(x_{n},y)-f_{t}(x,y)\right\vert \mathrm{d}y & \leq\limsup_{n\to\infty}\int_{D\backslash B_{\delta}(0)}f_{t}(x_{n},y)\mathrm{d}y+\int_{D\backslash B_{\delta}(0)}f_{t}(x,y)\mathrm{d}y\\
 & \leq2\int_{D\backslash B_{\delta}(0)}f_{t}(x,y)\mathrm{d}y.
\end{align*}
Note that the latter integral tends to zero as $\delta\to\infty$.
This proves the desired continuity of $x\mapsto f_{t}(x,\cdot)\in L^{1}(D)$.

(b) Let $t_{1}>t_{0}>0$ be arbitrary and consider two multi-indices
$\mathbf{q},\thinspace\tilde{\mathbf{q}}\in\mathbb{N}_{0}^{d}$ satisfying
\eqref{eq: multi index b}. Differentiating formally under the integral
in \eqref{eq:densitiy of p_t} with respect to $x_{J}$ and $y$,
we conclude the assertions of Theorem \ref{thm:existence of densities}.(b),
provided that we can show that
\begin{align}
\sup_{t\in[t_{0},t_{1}]}\int_{D}\|\psi_{J}(t,\mathrm{i}u)\|^{\sum_{j\in J}q_{j}}\|u_{J}\|^{\sum_{j\in J}\tilde{q}_{j}}\|u_{I}\|^{\sum_{i\in I}\tilde{q}_{i}}\mathrm{e}^{\mathrm{Re}(\phi(t,\mathrm{i}u))}\mathrm{d}u<\infty.\label{eq: bound for char function 1}
\end{align}
Now, we verify that \eqref{eq: bound for char function 1} is true.
Using $\psi_{J}(t,\mathrm{i}u)=\mathrm{i}\mathrm{e}^{\beta_{JJ}^{\top}t}u_{J}$
we obtain $\|\psi_{J}(t,\mathrm{i}u)\|\leq c_{0}(t_{1})\|u_{J}\|$
for $t\in[0,t_{1}]$ with $c_{0}(t_{1}):=\mathrm{e}^{t_{1}\|\beta_{JJ}\|_{HS}}$
where $\|\cdot\|_{HS}$ denotes the Hilbert-Schmidt norm. Hence we
obtain
\begin{align}
 & \int_{D}\|\psi_{J}(t,\mathrm{i}u)\|^{\sum_{j\in J}q_{j}}\|u_{J}\|^{\sum_{j\in J}\tilde{q}_{j}}\|u_{I}\|^{\sum_{i\in I}\tilde{q}_{i}}\mathrm{e}^{\mathrm{Re}(\phi(t,\mathrm{i}u))}\mathrm{d}u\nonumber \\
 & \qquad\qquad\le c_{0}(t_{1})^{|\mathbf{q}|}\int_{D}(1+\|u_{J}\|)^{|\mathbf{q}|+|\tilde{\mathbf{q}}|}(1+\|u_{I}\|)^{p}\mathrm{e}^{\mathrm{Re}(\phi(t,\mathrm{i}u))}\mathrm{d}u.\label{eq: not denoted estmates}
\end{align}
It follows easily from \eqref{eq: condition boundary} that there
exists some $\vartheta\in(0,1)$ small enough such that
\begin{equation}
m+p<\lambda(\vartheta)=\min_{i\in I}b_{i}/\widehat{\alpha}_{i,ii}(\vartheta),\label{eq: defi alpha hat}
\end{equation}
where $\widehat{\alpha}_{i,ii}$ is as in \eqref{eq: defi alpha hat-2}.
By Theorem \ref{thm: estimate characteristic function}, we obtain
\begin{equation}
\mathrm{e}^{\mathrm{Re}(\phi(t,\mathrm{i}u))}\le C(1+\|u_{I}\|)^{-\lambda(\vartheta)}\mathrm{e}^{-\delta\|u_{J}\|^{2}},\quad\|u\|\ge M,\ t\ge t_{0},\label{eq: second undenoted estimate}
\end{equation}
where $C,\thinspace\delta,\thinspace M>0$ are constants depending
on $t_{0}$ and $\vartheta$. Thus \eqref{eq: bound for char function 1}
readily follows from \eqref{eq: not denoted estmates}, \eqref{eq: defi alpha hat}
and \eqref{eq: second undenoted estimate}.

(c) Consider $t_{1}>t_{0}>0$ and let $\mathbf{q},\thinspace\tilde{\mathbf{q}}\in\mathbb{N}_{0}^{d}$
be two multi-indices satisfying \eqref{eq: multi index c}. Further,
let $K\subset D^{\mathrm{o}}$ be a compact set. The proof of this
assertion follows a similar approach to assertion (b). Indeed, by
formally differentiating under the integral in \eqref{eq:densitiy of p_t},
the assertion follows from the dominated convergence theorem, provided that we can show the existence of constants $c,\delta,M>0$ such that for all $(t,x,y)\in[t_{0},t_{1}]\times K\times D$
and all $u\in\mathbb{R}^{d}$ with $\|u\|\ge M$,
\begin{align}
 & \|\psi_{I}(t,\mathrm{i}u)\|^{\sum_{i\in I}q_{i}}\|u_{I}\|^{\sum_{i\in I}\tilde{q}_{i}}\|\psi_{J}(t,\mathrm{i}u)\|^{\sum_{j\in J}q_{j}}\|u_{J}\|^{\sum_{j\in J}\tilde{q}_{j}}\mathrm{e}^{\mathrm{Re}(\phi(t,\mathrm{i}u))+\mathrm{Re}(\langle x,\psi(t,\mathrm{i}u)\rangle)} \notag \\
 & \quad\le c\left(1+\|u_{I}\|\right)^{p-\lambda(\vartheta)}(1+\|u_{J}\|)^{|\mathbf{q}|+|\tilde{\mathbf{q}}|}\mathrm{e}^{-\delta\|u_{J}\|^{2}}. \label{estimate to apply dct}
\end{align}
Since
$K$ is a compact set with $K\cap\partial D=\emptyset$,
we find $\varepsilon>0$ such that $x_{i}\ge\varepsilon$ holds for
each $x\in K$ and $i\in\{1,\dots,m\}$. Then
\begin{align}
 & \|\psi_{I}(t,\mathrm{i}u)\|^{\sum_{i\in I}q_{i}}\|u_{I}\|^{\sum_{i\in I}\tilde{q}_{i}}\|\psi_{J}(t,\mathrm{i}u)\|^{\sum_{j\in J}q_{j}}\|u_{J}\|^{\sum_{j\in J}\tilde{q}_{j}}\mathrm{e}^{\mathrm{Re}(\phi(t,\mathrm{i}u))+\mathrm{Re}(\langle x,\psi(t,\mathrm{i}u)\rangle)}\nonumber \\
 & \quad\le\|\psi_{I}(t,\mathrm{i}u)\|^{\sum_{i\in I}q_{i}}\|u_{I}\|^{\sum_{i\in I}\tilde{q}_{i}}\|\psi_{J}(t,\mathrm{i}u)\|^{\sum_{j\in J}q_{j}}\|u_{J}\|^{\sum_{j\in J}\tilde{q}_{j}}\mathrm{e}^{\mathrm{Re}(\phi(t,\mathrm{i}u))+\varepsilon\sum_{i=1}^{m}\mathrm{Re}(\psi_{i}(t,\mathrm{i}u))}.\label{eq: bound for char function 2}
\end{align}
The troubling term can be estimated as follows:
\begin{align}
 & \|\psi_{I}(t,\mathrm{i}u)\|^{\sum_{i\in I}q_{i}}\mathrm{e}^{\varepsilon\sum_{i=1}^{m}\mathrm{Re}(\psi_{i}(t,\mathrm{i}u))}\nonumber \\
 & \quad\le\left(1+\|\psi_{I}(t,\mathrm{i}u)\|\right)^{\sum_{i\in I}q_{i}}\mathrm{e}^{\varepsilon\sum_{i=1}^{m}\mathrm{Re}(\psi_{i}(t,\mathrm{i}u))}\nonumber \\
 & \quad\leq2^{p}\left(1+\|\mathrm{Re}\left(\psi_{I}(t,\mathrm{i}u)\right)\|\right)^{p}\mathrm{e}^{\varepsilon\sum_{i=1}^{m}\mathrm{Re}(\psi_{i}(t,\mathrm{i}u))}\nonumber \\
 & \qquad+2^{p}\left(1+\|\mathrm{Im}\left(\psi_{I}(t,\mathrm{i}u)\right)\|\right)^{\sum_{i\in I}q_{i}}\mathrm{e}^{\varepsilon\sum_{i=1}^{m}\mathrm{Re}(\psi_{i}(t,\mathrm{i}u))}\nonumber \\
 & \quad\leq c_{1}\sum_{_{i=1}}^{m}\left(1+\left|\mathrm{Re}\left(\psi_{i}(t,\mathrm{i}u)\right)\right|\right)^{p}\mathrm{e}^{\varepsilon\mathrm{Re}(\psi_{i}(t,\mathrm{i}u))}+c_{1}\left(1+\|\mathrm{Im}\left(\psi_{I}(t,\mathrm{i}u)\right)\|\right)^{\sum_{i\in I}q_{i}},\label{eq: first bound for trouble term}
\end{align}
where $c_{1}>0$ is a constant. Note that the function
\begin{equation}
(-\infty,0]\ni r\mapsto(1+|r|)^{p}e^{\varepsilon r}\quad\mbox{is bounded.}\label{eq: boundedness of a real function}
\end{equation}
By \eqref{eq: addtion condtion on mu_i} and \cite[p.108, (C.5)]{MR3084047},
we can find a constant $c_{2}>0$ such that for all $t\in[t_{0},t_{1}]$
and $u\neq0$,
\[
\|G(t,\mathrm{i}u)\|^{2}\leq\exp\left(\frac{c_{2}t}{\|u\|}\right),
\]
where $G(t,\mathrm{i}u)$ is as in \eqref{eq: defi of F and G}. This
gives
\begin{align}
\left|\mathrm{Im}\left(\psi_{i}(t,\mathrm{i}u)\right)\right| & =\Vert u\Vert\left|G_{i}(t\Vert u\Vert,\mathrm{i}u)\right|\le\Vert u\Vert\exp\left(\frac{c_{2}t_{1}}{2}\right).\label{eq: second bound for trouble term}
\end{align}
Combining \eqref{eq: first bound for trouble term}, \eqref{eq: boundedness of a real function},
and \eqref{eq: second bound for trouble term} gives
\begin{align*}
\|\psi_{I}(t,\mathrm{i}u)\|^{\sum_{i\in I}q_{i}}\mathrm{e}^{\varepsilon\sum_{i=1}^{m}\mathrm{Re}(\psi_{i}(t,\mathrm{i}u))} & \le c_{3}\left(1+\|u\|\right)^{\sum_{i\in I}q_{i}}.
\end{align*}
In view of \eqref{eq: bound for char function 2}, we finally get
\begin{align*}
 & \|\psi_{I}(t,\mathrm{i}u)\|^{\sum_{i\in I}q_{i}}\|u_{I}\|^{\sum_{i\in I}\tilde{q}_{i}}\|\psi_{J}(t,\mathrm{i}u)\|^{\sum_{j\in J}q_{j}}\|u_{J}\|^{\sum_{j\in J}\tilde{q}_{j}}\mathrm{e}^{\mathrm{Re}(\phi(t,\mathrm{i}u))+\mathrm{Re}(\langle x,\psi(t,\mathrm{i}u)\rangle)}\\
 & \quad\le c_{3}\left(1+\|u\|\right)^{\sum_{i\in I}q_{i}}\|u_{I}\|^{\sum_{i\in I}\tilde{q}_{i}}\|\psi_{J}(t,\mathrm{i}u)\|^{\sum_{j\in J}q_{j}}\|u_{J}\|^{\sum_{j\in J}\tilde{q}_{j}}\mathrm{e}^{\mathrm{Re}(\phi(t,\mathrm{i}u))}\\
 & \quad\le c_{4}\left(1+\|u_{I}\|\right)^{p}(1+\|u_{J}\|)^{|\mathbf{q}|+|\tilde{\mathbf{q}}|}\mathrm{e}^{\mathrm{Re}(\phi(t,\mathrm{i}u))}
\end{align*}
for another constant $c_{4}>0$. In view of \eqref{eq: second undenoted estimate}, we thus arrive at \eqref{estimate to apply dct}
and the rest of the proof goes then exactly as in part (b). This completes
the proof.
\end{proof}

\subsection{Existence and regularity of the invariant measure}

Below we prove the existence and regularity of the density for the
invariant measure.
\begin{lem}
Suppose that the same conditions as in Theorem \ref{thm:exp ergo}
are satisfied. Then the unique invariant measure $\pi$ has a density
$f^{\pi}$ of class $C_{0}^{p}(D)$. Moreover, for each multi-index
$\mathbf{q}\in\mathbb{N}_{0}^{d}$ satisfying $\sum_{i\in I}q_{i}\leq p$,
the derivative $\partial^{|\mathbf{q}|}\left(f^{\pi}\right)/\partial y_{1}^{q_{1}}\dots\partial y_{d}^{q_{d}}$
is given by
\[
\frac{\partial^{|\mathbf{q}|}f^{\pi}}{\partial y_{1}^{q_{1}}\dots\partial y_{d}^{q_{d}}}(y)=(-i)^{|\mathbf{q}|}\int_{\mathbb{R}^{d}}\prod_{k=1}^{d}\left(u_{k}\right)^{q_{k}}\mathrm{e}^{-\langle y,\mathrm{i}u\rangle}\mathrm{e}^{\phi(t,\mathrm{i}u)}\frac{\mathrm{d}u}{(2\pi)^{d}},\qquad y\in D.
\]
\end{lem}

\begin{proof}
Similarly as in the proof of Theorem \ref{thm:existence of densities},
we first find small enough $\vartheta\in(0,1)$ such that
\begin{equation}
m+p<\lambda(\vartheta)=\min_{i\in I}b_{i}/\widehat{\alpha}_{i,ii}(\vartheta),\label{eq: defi alpha hat-1}
\end{equation}
where $\widehat{\alpha}_{i,ii}(\vartheta)$ is as in \eqref{eq: defi alpha hat-2},
and then apply Theorem \ref{thm: estimate characteristic function}
to get
\[
\mathrm{e}^{\mathrm{Re}(\phi(t,\mathrm{i}u))}\le C(1+\|u_{I}\|)^{-\lambda(\vartheta)}\mathrm{e}^{-\delta\|u_{J}\|^{2}},\quad\|u\|\ge M,\ t\ge1,
\]
where $C,\thinspace\delta,\thinspace M>0$ are constants depending
on $\vartheta$. So
\begin{align*}
\left|\int_{D}\mathrm{e}^{\langle\mathrm{i}u,\xi\rangle}P_{t}(x,\mathrm{d}\xi)\right| & \leq C(1+\|u_{I}\|)^{-\lambda(\vartheta)}\mathrm{e}^{-\delta\|u_{J}\|^{2}}
\end{align*}
holds for all $x\in D$, $u\in\mathbb{R}^{d}$ with $\|u\|\geq M$,
and $t\geq1$. Following \cite[Theorem 2.7]{2018arXiv181205402J}
it holds
\[
\int_{D}\mathrm{e}^{\langle\mathrm{i}u,\xi\rangle}\pi(\mathrm{d}\xi)=\lim_{t\to\infty}\int_{D}\mathrm{e}^{\langle\mathrm{i}u,\xi\rangle}P_{t}(x,\mathrm{d}\xi),\qquad u\in\mathbb{R}^{d},\ \ x\in D,
\]
and hence we obtain, for $u\in\mathbb{R}^{d}$ with $\|u\|\geq M$,
\[
\left|\int_{D}\mathrm{e}^{\langle\mathrm{i}u,\xi\rangle}\pi(\mathrm{d}\xi)\right|\leq C(1+\|u_{I}\|)^{-\lambda(\vartheta)}\mathrm{e}^{-\delta\|u_{J}\|^{2}}.
\]
In view of \eqref{eq: defi alpha hat-1}, the existence and differentiability
of the density $f^{\pi}$ follows by classical properties of the Fourier
transform.
\end{proof}

\section{Exponential ergodicity in the total variation norm}
\label{sec:exp ergodicity}

In this section we prove the exponential ergodicity
statement in Theorem \ref{thm:exp ergo}. For this purpose, we first
prove a similar result for the affine process with transition semigroup
$(Q_{t})_{t\geq0}$ and admissible parameters $(a,\alpha,b,\beta,\nu=0,\mu)$,
i.e.,
\begin{equation}
\int_{D}\mathrm{e}^{\langle u,\xi\rangle}Q_{t}(x,\mathrm{d}\xi)=\exp\left(\int_{0}^{t}\langle\psi(s,u),a\psi(s,u)\rangle\mathrm{d}s+\langle b,\int_{0}^{t}\psi(s,u)\mathrm{d}s\rangle+\langle x,\psi(t,u)\rangle\right)\label{eq:definition of Q_t}
\end{equation}
for all $u\in\mathcal{U}$, where $\psi(t,u)$ is obtained from \eqref{eq:riccati eq for psi}.
The assertion is then deduced by a convolution argument similar to
\cite{2019arXiv190105815F,2018arXiv181205402J}. According to Theorem
\ref{thm:characterization of affine processes}, the generator $\mathcal{A}_{Q}$
of $(Q_{t})_{t\geq0}$ is given by
\begin{equation}
\mathcal{A}_{Q}f(x)=\mathcal{A}_{Q}^{0}f(x)+\mathcal{A}_{Q}^{1}f(x),\label{eq:generator of Q}
\end{equation}
for all $x\in D$ and defined for every $f\in C_{c}^{2}(D)$, where
\begin{align}
\mathcal{A}_{Q}^{0}f(x) & =\langle b,\nabla f(x)\rangle+\sum_{k,l=1}^{d}a_{kl}\frac{\partial^{2}f(x)}{\partial x_{k}\partial x_{l}},\nonumber\\
\mathcal{A}_{Q}^{1}f(x) & =\langle\beta x,\nabla f(x)\rangle+\sum_{k,l=1}^{d}\left(\sum_{i=1}^{m}\alpha_{i,kl}x_{i}\right)\frac{\partial^{2}f(x)}{\partial x_{k}\partial x_{l}}\nonumber \\
 & \ \ \ +\sum_{i=1}^{m}x_{i}\int_{D\backslash\lbrace0\rbrace}\left(f(x+\xi)-f(x)-\langle\xi,\nabla f(x)\rangle\right)\mu_{i}(\mathrm{d}\xi).\nonumber
\end{align}
Proceeding as in \cite[p.~10]{2018arXiv181205402J}, for $\beta$
whose eigenvalues have strictly negative real-parts, we define the
following norms
\[
\Vert x_{I}\Vert_{I}:=\sqrt{\langle x_{I},x_{I}\rangle_{I}}=\sqrt{\langle x_{I},M_{I}x_{I}\rangle}\quad\text{and}\quad\Vert x_{J}\Vert_{J}:=\sqrt{\langle x_{J},x_{J}\rangle_{J}}=\sqrt{\langle x_{J},M_{J}x_{J}\rangle},
\]
where
\[
M_{I}:=\int_{0}^{\infty}\mathrm{e}^{t\beta_{II}^{\top}}\mathrm{e}^{t\beta_{II}}\mathrm{d}t\quad\text{and}\quad M_{J}:=\int_{0}^{\infty}\mathrm{e}^{t\beta_{JJ}^{\top}}\mathrm{e}^{t\beta_{JJ}}\mathrm{d}t.
\]
Since $\beta_{II}$ and $\beta_{JJ}$ have only eigenvalues with strictly
negative real-parts, the matrices $M_{I}$ and $M_{J}$ are well-defined.
Note that both $M_{I}$ and $M_{J}$ are symmetric positive definite
matrices.

\begin{prop}\label{prop:lyapunov function} Assume $d=m+n\geq1$
and let $(Q_{t})_{t\geq0}$ be the affine semigroup given by \eqref{eq:definition of Q_t}
with admissible parameters $(a,\alpha,b,\beta,\nu=0,\mu)$ such that
$\beta$ has only eigenvalues with strictly negative real-parts. Let
$V_{\varepsilon}\in C^{2}(D)$ be given by
\[
V_{\varepsilon}(x)=\left(1+\langle x_{I},x_{I}\rangle_{I}+\varepsilon\langle x_{J},x_{J}\rangle_{J}\right)^{1/2},\quad x\in\mathbb{R}_{\geq0}^{m}\times\mathbb{R}^{n},
\]
where $\varepsilon>0$. Then $V_{\varepsilon}$ belongs to the domain
of the extended generator $\mathcal{A}_{Q}$. If $\varepsilon>0$
is small enough, then there exist positive constants $c$ and $C$
such that
\begin{equation}
\mathcal{A}_{Q}V_{\varepsilon}(x)\leq -cV_{\varepsilon}(x)+C,\quad\text{for all }x\in D.\label{eq:lyapunov estimate}
\end{equation}
\end{prop}

\begin{proof} Following the arguments given in the proof of \cite[Proposition 5.1 (a)]{2019arXiv190105815F},
we conclude that $V_{\varepsilon}$ belongs to the domain of the extended
generator and that $\mathcal{A}_{Q}V_{\varepsilon}$ is given as in
\eqref{eq:generator of Q}. It remains to prove \eqref{eq:lyapunov estimate}.
Now, it follows from the particular form of $V_{\varepsilon}$ that
there exists a constant $c_{1}>0$ such that $|\mathcal{A}_{Q}^{0}V_{\varepsilon}(x)|\leq c_{1}$
for all $x\in D$. Next, note that there exist constants $c_{2},\thinspace c_{3}>0$
and sufficiently small $\varepsilon\in(0,1/2)$ such that
\begin{align}
\mathcal{A}_{Q}^{1}V_{\varepsilon}(x)\leq c_{2}-c_{3}(\|x_{I}\|_{I}^{2}+\varepsilon\|x_{J}\|_{J}^{2})^{1/2},\qquad\|x\|>2.\label{eq: lyapunov estimate}
\end{align}
Indeed, if $m\geq1$ and $n\geq1$, then this follows from the inequalities
shown in the proof of \cite[Lemma 3.4]{2018arXiv181205402J}; the
case $m\geq1$, $n=0$ follows by similar arguments to \cite[Proposition 3.7]{2018arXiv181205402J};
while the case $m=0$, $n\geq1$ describes a L\'evy driven OU-process
and can be shown by similar (but essentially simpler) arguments to
the previous two cases. Using \eqref{eq: lyapunov estimate} combined
with $|\mathcal{A}_{Q}^{1}V_{\varepsilon}(x)|\leq c_{4}$, $x\in D$,
where $c_{4}>0$ is some constant, we can easily show that
\begin{align*}
\mathcal{A}_{Q}^{1}V_{\varepsilon}(x) & \leq c_{5}-c_{6}\sqrt{\varepsilon}V_{\varepsilon}(x)
\end{align*}
where $c_{5},\thinspace c_{6}>0$ are some constants. Combining both estimates
for $\mathcal{A}_{Q}^{0}V_{\varepsilon}(x)$ and $\mathcal{A}_{Q}^{1}V_{\varepsilon}(x)$
readily yields \eqref{eq:lyapunov estimate}. \end{proof}

From the
Lyapunov estimate we shall deduce a contraction estimate in total
variation distance for the transition semigroup $(Q_{t})_{t\geq0}$.
\begin{prop}\label{prop:strong feller property} Let $(Q_{t})_{t\geq0}$
be the affine semigroup given by \eqref{eq:definition of Q_t} with
admissible parameters $(a,\alpha,b,\beta,\nu=0,\mu)$ such that $\mathcal{K}$
has full rank, $\min_{i\in\{1,\dots,m\}}\alpha_{i,ii}>0$, and $\beta$
has only eigenvalues with strictly negative real-parts. Further, suppose
that $m<\min_{i\in I}b_{i}\alpha_{i,ii}^{-1}$. Then for every $M>0$
there exists $h>0$ and $\delta\in(0,2)$ such that
\[
\left\Vert Q_{h}(x,\cdot)-Q_{h}(y,\cdot)\right\Vert _{TV}\leq2-\delta,\quad\text{for all }x,\thinspace y\in D\text{ with }\Vert x\Vert,\thinspace\Vert y\Vert\leq M.
\]
\end{prop}

The proof of Proposition \ref{prop:strong feller property} goes along
the very same lines as in the proof of \cite[Proposition 5.3]{2019arXiv190805473F},
see part (ii) therein. We omit the details here. We are ready to proof
our second main result.

\begin{proof}[Proof of Theorem \ref{thm:exp ergo}] Denote by $(Q_{t})_{t\geq0}$
the transition semigroup given by \eqref{eq:definition of Q_t} and
let $(R_{t})_{t\geq0}$ be the transition semigroup given by
\begin{align*}
 & \int_{D}\mathrm{e}^{\langle u,\xi\rangle}R_{t}(x,\mathrm{d}\xi)\nonumber \\
 & \quad=\exp\left(\int_{0}^{t}\int_{D\backslash\{0\}}\left(\mathrm{e}^{\langle\psi(s,u),\xi\rangle}-1-\langle\psi_{J}(s,u),\xi_{J}\rangle\mathbbm{1}_{\lbrace\Vert\xi\Vert\leq1\rbrace}(\xi)\right)\nu(\mathrm{d}\xi)\mathrm{d}s+\langle x,\psi(t,u)\rangle\right),
\end{align*}
for all $u\in\mathcal{U}$, where $\psi(t,u)$ is given in \eqref{eq:riccati eq for psi}.
Then $R_{t}(x,\mathrm{d}\xi)$ has admissible parameters $(a=0,\alpha,b=0,\beta,\nu,\mu)$.
Using \eqref{eq:affine property}, we obtain
\[
\int_{D}\mathrm{e}^{\langle u,\xi\rangle}P_{t}(x,\mathrm{d}\xi)=\int_{D}\mathrm{e}^{\langle u,\xi\rangle}Q_{t}(x,\mathrm{d}\xi)\cdot\int_{D}\mathrm{e}^{\langle u,\xi\rangle}R_{t}(0,\mathrm{d}\xi),\quad t\geq0,\thinspace u\in\mathcal{U},
\]
implying that $P_{t}(x,\cdot)=Q_{t}(x,\cdot)\ast R_{t}(0,\cdot)$
for all $t\geq0$ and $x\in D$, where `$\ast$' denotes the usual
convolution between measures. As a consequence of \cite[Theorem 4.1]{Hairer10convergenceof}
(see also \cite[Corollary 2.8.3 and Theorem 3.2.3]{MR3791835}), combining
Propositions \ref{prop:lyapunov function} and \ref{prop:strong feller property},
yields
\begin{equation}
\left\Vert Q_{t}(x,\cdot)-Q_{t}(y,\cdot)\right\Vert _{TV}\leq c\min\left\lbrace 1,\mathrm{e}^{-ct}\left(1+\Vert x\Vert+\Vert y\Vert\right)\right\rbrace ,\quad t\geq0,\thinspace x\in D,\label{eq:key estimate for exp ergodicity}
\end{equation}
with some constant $c>0$. In the following we let $H$ be any coupling\footnote{A \emph{coupling} $H$ of two Borel probability measures $(\varrho,\widetilde{\varrho})$
on $D$ is a again Borel probability measure on $D\times D$ which
has marginals $\varrho$ and $\widetilde{\varrho}$, that is, for
bounded Borel measurable functions $f$ and $g$ on $D$ it holds
\[
\int_{D\times D}\left(f(x)+g(\widetilde{x})\right)H(\mathrm{d}x,\mathrm{d}\widetilde{x})=\int_{D}f(x)\varrho(\mathrm{d}x)+\int_{D}g(x)\widetilde{\varrho}(\mathrm{d}x).
\]
} of the Dirac measure $\delta_{x}$ concentrated in $x$ and the invariant
distribution $\pi$. Let $C>0$ be a generic constant that may vary
from line to line. From \cite[Theorem 4.8]{MR2459454} combined with
the invariance of $\pi$ and an obvious extension of \cite[Lemma 2.3]{2019arXiv190202833F},
we get
\begin{align*}
\Vert P_{t}(x,\cdot)-\pi(\cdot)\Vert_{TV} & \leq\int_{D\times D}\left\Vert P_{t}(x,\cdot)-P_{t}(y,\cdot)\right\Vert _{TV}H(\mathrm{d}x,\mathrm{d}y)\\
 & \leq\int_{D\times D}\Vert Q_{t}(x,\cdot)-Q_{t}(y,\cdot)\Vert_{TV}H(\mathrm{d}x,\mathrm{d}y)\\
 & \leq C\int_{D\times D}\left(1\wedge\mathrm{e}^{-ct}\left(1+\Vert x\Vert+\Vert y\Vert\right)\right)H(\mathrm{d}x,\mathrm{d}y)\\
 & \leq C\int_{D\times D}\log\left(1+\mathrm{e}^{-ct}\left(1+\Vert x\Vert+\Vert y\Vert\right)\right)H(\mathrm{d}x,\mathrm{d}y),
\end{align*}
where the last two inequalities follow from \eqref{eq:key estimate for exp ergodicity}
and the trivial inequality $1\wedge c_{1}\leq\log(2)^{-1}\log(1+c_{1})$
for all $c_{1}>0$. In \cite[Lemma 8.5]{2019arXiv190105815F} the
following inequality was shown
\[
\log(1+c_{2}\cdot c_{3})\leq c\min\left\lbrace c_{2},\log(1+c_{3})\right\rbrace +c\log(1+c_{3}),
\]
for any $c_{2},\thinspace c_{3}\geq0$, where $c>0$ is a constant.
So
\begin{align*}
\Vert P_{t}(x,\cdot)-\pi(\cdot)\Vert_{TV} & \leq C\int_{D\times D}\min\left\lbrace \mathrm{e}^{-ct},\log\left(1+\Vert x\Vert+\Vert y\Vert\right)\right\rbrace H(\mathrm{d}x,\mathrm{d}y)\\
 & \quad+C\mathrm{e}^{-ct}\int_{D\times D}\log\left(1+\Vert x\Vert+\Vert y\Vert\right)H(\mathrm{d}x,\mathrm{d}y)\\
 & \leq C\min\left\lbrace \mathrm{e}^{-ct},\int_{D\times D}\log\left(1+\Vert x\Vert+\Vert y\Vert\right)H(\mathrm{d}x,\mathrm{d}y)\right\rbrace \\
 & \quad+C\mathrm{e}^{-ct}\int_{D\times D}\log\left(1+\Vert x\Vert+\Vert y\Vert\right)H(\mathrm{d}x,\mathrm{d}y)\\
 & \leq C\mathrm{e}^{-ct}\left(1+\int_{D\times D}\log\left(1+\Vert x\Vert+\Vert y\Vert\right)H(\mathrm{d}x,\mathrm{d}y)\right).
\end{align*}
Taking $H$ as the optimal coupling of $(\delta_{x},\pi)$ and using
the subadditivity of $\log$, we see that
\[
\Vert P_{t}(x,\cdot)-\pi(\cdot)\Vert_{TV}\leq C\mathrm{e}^{-ct}\left(1+\log\left(1+\Vert x\Vert\right)+\int_{D}\log\left(1+\Vert y\Vert\right)\pi(\mathrm{d}y)\right).
\]
Notice that the integral on the right-hand side is indeed finite by
\cite[Theorem 1.5]{2019arXiv190105815F}. \end{proof}

\vspace{0.5cm}

\textbf{Acknowledgements.} We would like to thank the insurance company
Debeka in Koblenz for their keen interest and partially financing
the research that was carried out in the present article. The authors
would also like to thank Eberhard Mayerhofer for helpful comments
on the topic of this work. The research of Peng Jin is supported by the National Natural Science Foundation of China (No. 11861029).

\bibliographystyle{amsplain}
\providecommand{\bysame}{\leavevmode\hbox to3em{\hrulefill}\thinspace}
\providecommand{\MR}{\relax\ifhmode\unskip\space\fi MR }
\providecommand{\MRhref}[2]{%
  \href{http://www.ams.org/mathscinet-getitem?mr=#1}{#2}
}
\providecommand{\href}[2]{#2}

\end{document}